\documentclass{amsart}

\usepackage{amsmath}
\usepackage{url}
\usepackage{amsthm}
\usepackage{amsfonts}
\usepackage{hyperref}
\usepackage{graphicx}
\usepackage{subfig}

\theoremstyle{plain}
  \newtheorem{theorem}{Theorem}
  \newtheorem{lemma}{Lemma}
\theoremstyle{definition}
  \newtheorem{definition}[subsection]{Definition}
  \newtheorem{remark}[subsection]{Remark}

\newcommand{\T}{\mathcal{T}}
\newcommand{\F}{\mathcal{F}}
\newcommand{\R}{\mathbb{R}}

\newcommand{\N}{\mathbb{N}}

\begin{document}
\title[]{Free Boundary Minimal Surfaces in the Unit Ball With Low Cohomogeneity}
\author{Brian Freidin, Mamikon Gulian and Peter McGrath}
\date{}
\maketitle
\begin{abstract}
We study free boundary minimal surfaces in the unit ball of low cohomogeneity.  For each pair of positive integers $(m,n)$ such that $m, n >1$ and $m+n\geq 8$, we construct a free boundary minimal surface $\Sigma_{m, n} \subset B^{m+n}$(1) invariant under $O(m)\times O(n)$.  When $m+n<8$, an instability of the resulting equation allows us to find an infinite family $\{\Sigma_{m,n, k}\}_{k\in \N}$ of such surfaces.  In particular, $\{\Sigma_{2, 2, k}\}_{k\in \N}$ is a family of solid tori which converges to the cone over the Clifford Torus as $k$ goes to infinity.  These examples indicate that a smooth compactness theorem for Free Boundary Minimal Surfaces due to Fraser and Li does not generally extend to higher dimensions. 

For each $n\geq 3$, we prove there is a unique nonplanar $SO(n)$-invariant free boundary minimal surface (a ``catenoid'') $\Sigma_n \subset B^n(1)$.  These surfaces generalize the ``critical catenoid'' in $B^3(1)$ studied by Fraser and Schoen.

\end{abstract}

\section{Introduction}
There has been recent interest in studying free boundary minimal surfaces in the unit ball.  Fraser and Schoen proved (Theorem 5.4, \cite{FS1}) that any free boundary minimal surface $\Sigma^2 \subset B^3(1)$ has area at least $\pi$. Brendle \cite{Brendle} extended this result to free boundary $\Sigma^k \subset B^n(1)$, and Freidin and McGrath \cite{FM} recently proved the analogous result for geodesic balls in Hyperbolic space $H^n$.  Along slightly different lines, Nitsche \cite{Nitsche} proved that the only free boundary minimal disks in $B^3(1)$ are equatorial disks; Souam \cite{Souam} extended this result to hold for balls in 3-dimensional space forms, and Fraser and Schoen further showed analogous rigidity holds for 2-disks of higher codimension.

Despite these results, there are few explicitly known examples of free boundary minimal surfaces in the unit ball.  As a consequence of their work on the Steklov Eigenvalue Problem, Fraser and Schoen have exhibited a family of free boundary minimal surfaces in  $B^3(1)$ with genus $0$ and any number of boundary components (Theorem 1.6, \cite{FS2}); such surfaces have recently been constructed using gluing methods \cite{Pacard 2} when the number of boundary components is large.  Little is known about the existence of free boundary minimal surfaces of higher genus.  In higher dimensions, the landscape is even more sparse, and to the authors' knowledge, no nontrivial free boundary surfaces in $B^n(1)$, $n>3$ have been constructed to date.  One purpose of this paper is to construct many new such surfaces.  These examples are invariant under groups of cohomogeneity one or two (to use the terminology of \cite{HL}). Imposing an ansatz of such symmetry reduces the minimal surface equation and associated free boundary condition to a more easily analyzed nonlinear second order ODE and associated boundary condition on an appropriate orbit space.  Hsiang developed these methods, which he called ``Equivariant Differential Geometry'' to carry out various constructions (cf. \cite{HL}, \cite{H1}, \cite{H2}, \cite{H3}).  In particular, Hsiang proved the existence of non-equatorial embedded minimal hyperspheres in $S^n$ for several $n>3$ (settling the so-called ``Spherical Bernstein Conjecture''), and the existence of infinitely many noncongruent, closed, embedded minimal surfaces in $S^n$ for $n\geq 3$.  Our main construction is:
\begin{theorem}
\label{thm: construct} Suppose $(m, n) \in \N^2$ and $m, n >1$.
\begin{enumerate}
\item If $m+n<8$, there exists an infinite family $\{ \Sigma_{m,n, k}\}_{k\in \N}$ of mutually noncongruent $O(m)\times O(n)$-invariant free boundary minimal hypersurfaces in $B^{m+n}(1)$, each of which is homeomorphic to $B^{m}\times S^{n-1}$.  As $k\rightarrow \infty$, the surfaces $\Sigma_{m, n, k}$ converge in $C^0(\R^{m+n})\bigcap C^\infty(\R^{m+n}\setminus\{0\})$ to the minimal cone $\mathcal{C}_{m,n}$ over $S^{m-1}\left(\sqrt{\frac{m-1}{m+n-2}} \right)\times S^{n-1}\left( \sqrt{\frac{n-1}{m+n-2}}\right) \subset S^{m+n-1}(1)$.  
\item If $m+n\geq 8$, there exists an $O(m)\times O(n)$-invariant free boundary hypersurface $\Sigma_{m, n}\subset B^{m+n}(1)$ homeomorphic to $S^{m-1}\times S^{n-1}\times [0,1]$.
\end{enumerate}
\end{theorem}
Probably the most interesting case of Theorem \ref{thm: construct} is when $m=n=2$ where $\Sigma_{2, 2, k}\subset B^4(1)$ is a family of free boundary solid tori.  As $k\rightarrow \infty$, $\Sigma_{2, 2, k}$ converges to the cone over the Clifford Torus.  A similar phenomena was observed by Hsiang, who constructed a family of embedded $O(2)\times O(2)$-invariant minimal surfaces in $S^4$, each homeomorphic to $S^3$ (Theorems 1 and 2, \cite{H1}).  In particular, the examples from our Theorem \ref{thm: construct}, part (1) illustrate a lack of smooth compactness for free boundary minimal surfaces in $B^n(1)$ for $n>3$, in marked contrast to the compactness theorem \cite{Fraser-Li} proved by Fraser and Li in dimension three.

The proof of Theorem \ref{thm: construct} requires an understanding of the space of complete, $O(m)\times O(n)$-invariant minimal surfaces in $\R^{m+n}$, which has been studied by several authors.  In the early 1980s, Hsiang classified $O(m)\times O(n)$-invariant hypersurfaces in $\R^{m+n}$ with constant mean curvature \cite{H2} and there (remark (ii) in \cite{H2} below the proof of proposition 2') refers to a forthcoming paper describing the minimal case.  However, to the authors' knowledge, this paper never appeared.  Alencar \cite{Alencar1} considered the $m=n$ case and Ilmanen (\cite{Ilmanen}, p.44) later described properties of the general case.  A detailed analysis of the general case was carried out by Alencar et. al. in \cite{Alencar2}.  In section \ref{mn}, we review relevant properties of such minimal surfaces and prove Theorem \ref{thm: construct} by finding surfaces which ``fit'' inside $B^{m+n}(1)$ to satisfy the free boundary condition.  

The reason for a dichotomy between the cases $m+n< 8$ and $m+n\geq 8$ is due to a stability property of the associated minimal surface equation.  When $m+n<8$, the zeros of a relevant vector field on a space of parameters have focal singularities which force solutions to exhibit an oscillatory behavior.  On the other hand, when $m+n\geq 8$, the zeros have nodal singularities and the associated solution curves do not oscillate.  Similar behavior is present in other settings; e.g., for radially symmetric harmonic maps  $u: B^n\rightarrow S^n$ (c.f. example 2.2 in \cite{SU}) and the constructions of Hsiang discussed above.  

Starting in section \ref{forces} we consider free boundary surfaces in $B^n(1)$ invariant under $O(n)$.  The ansatz of full rotational symmetry is more restrictive, and we prove
\begin{theorem}\label{thm: catenoid}
Modulo isometries, there is a unique free boundary $n$-catenoid $\Sigma_n \subset B^{n+1}(1)$.  
\end{theorem}

When $n=2$, Theorem \ref{thm: catenoid} appears to be a folklore result - in particular, it is stated as a fact without reference in \cite{Nitsche}, page 2.  The authors have, however, been unable to find a proof in the literature; it appears here in the more general context of Theorem \ref{thm: catenoid}.  See also remark \ref{rmk: crit cat} for a more geometric proof in the 2-dimensional case.

After recalling well-known properties of the $n$-catenoid and notions of torque balancing in Sections \ref{forces} and \ref{cat properties}, we prove Theorem \ref{thm: catenoid} in Section \ref{main proof}.

The third author would like to thank his thesis advisor, Nikolaos Kapouleas, for suggesting the problem leading to Theorems 1 and 2 and for suggesting a torque balancing argument to prove Lemma \ref{lemma: coaxial}, which superseded his original, less geometric proof.

\section{Notation and Conventions}
\label{section: notation}
\begin{definition}
\label{def: fbc}
We say a smooth submanifold $\Sigma \subset B^k(1)$ is a \emph{free boundary minimal surface in $B^k(1)$} if $\Sigma$ is minimal, $\partial \Sigma \subset \partial B^k(1)$, and $\Sigma$ intersects $\partial B^k(1)$ orthogonally along $\partial \Sigma$.
\end{definition}
If $\eta$ is the outward pointing unit conormal to $\Sigma$ along $\partial \Sigma$ and $X$ is the position vector field, the free boundary condition implies that $\langle X, \eta\rangle = 1$.

Let $(\R^k, \delta_{ij})$ be Euclidean space and let $O(k)$ be the group of isometries of $\R^k$ preserving the origin.  Let $G$ be a Lie subgroup of $O(k)$ and let $\Pi: \R^k \rightarrow \R^k/G$ be the natural projection.  We say a submanifold $\Sigma \subset \R^k$ is $G$-invariant if for each $g\in G$, $g p\in \Sigma$ for all $p\in \Sigma$.  As in Theorem 2 in \cite{HL}, there exists an \emph{orbital metric} $g_{\Pi}$ on $\R^k /G$ such that $G$-invariant minimal surfaces in $(\R^k, \delta_{ij})$ correspond to minimal surfaces in $(\R^k/G, g_{\Pi})$ under the projection $\Pi$.  
In this paper, we consider two cases:\newline
\textbf{Case 1:}\newline $k = m+n$ for $m, n>1$ and $G = O(m)\times O(n)$ acts on $\R^k$ by the product action.  
We identify the orbit space $\R^k/G$ with the closed first quadrant
\begin{align*}
Q = \{ (x, y) \in \R^2: x\geq 0, y\geq 0\}.
\end{align*}
Using the quotient map $\Pi: \R^m \times \R^n \rightarrow Q$ defined by $\Pi(X, Y) = (|X|, |Y|)$, the inverse image of $(x,y)\in Q$ under $\Pi$ is a product manifold $S^m(x)\times S^{n}(y)$.  The orbital metric is (up to a constant multiplicative factor)
\begin{align}
\label{(m,n) metric}
g_{m, n}  = x^{m-1}y^{n-1} ( dx^2+ dy^2).
\end{align}
\textbf{Case 2:}\newline $k=n+1$, $G=O(n)$, and $G$ acts in the standard way on the last $n$ coordinates of $\R^{k}$.  
We identify the orbit space $\R^{k}/G$ with the closed half space 
\begin{align*}
H = \{ (x, y)\in \R^2: y\geq 0\}.
\end{align*}
Using the quotient map $\Pi: \R^k = \R\times \R^n \rightarrow H$ defined by $\Pi(X, Y) = (X, |Y|)$, the inverse image of $(x,y)\in H$ under $\Pi$ is a sphere $\{x\}\times S^n(y)$.  The orbital metric is (up to a constant multiplicative factor)
\begin{align}
\label{n metric}
g_{n} = x^{n-1}(dx^2+dy^2).
\end{align}

\begin{definition}
\label{def: fb curve}
We say an immersed curve $\gamma : I \rightarrow \R^k/G$ is a \emph{profile curve} if $\gamma$ is a geodesic with respect to the orbital metric.  For convenience, we nonetheless parametrize profile curves with respect to the standard Euclidean arc length on $\R^k/G$.  We say a profile curve $\gamma$ is a \emph{free boundary profile curve} if for all $t$ such that $|\gamma(t)|=1$, $\gamma(t) = \gamma'(t)$.
\end{definition}
Clearly, free boundary profile curves in $\R^k/G$ correspond to $G$-invariant free boundary minimal surfaces in $B^{k}(1)$.


\section{$O(m)\times O(n)$-invariant Minimal Surfaces}
\label{mn}
The Euler-Lagrange equation for the arclength integral \eqref{(m,n) metric} is
\begin{align}
\label{eqn el}
-x''(t)y'(t)+y''(t)x'(t)+ \frac{(m-1)y'(t) y(t) - (n-1)x'(t) x(t)}{x(t)y(t)} = 0.
\end{align}
  
Let $\ell_{m, n}$ be the line in $Q$ given by the equation $y = \sqrt{\frac{n-1}{m-1}} x$.  It is straightforward that a parametrization of $\ell_{m,n}$ satisfies equation \eqref{mn}.  Under the inverse image of $\Pi$, $\ell_{m,n}$ corresponds to the cone $\mathcal{C}_{m,n}$ over the ``Clifford'' type minimal surface $S^{m-1}\left(\sqrt{\frac{m-1}{m+n-2}} \right)\times S^{n-1}\left( \sqrt{\frac{n-1}{m+n-2}}\right) \subset S^{m+n-1}(1)$.

Given an arc length parametrized profile curve $\gamma(t)$, we define a radial parameter $r(t) = |\gamma(t)|$ and angular parameters $\varphi(t), \theta(t)$ by requesting that
\begin{align}
\label{eqn angles}
\begin{split}
r(t)\cos\left(\varphi(t)\right) &= x(t) \quad \text{and}\quad r(t)\sin\left(\varphi(t)\right) = y(t),\\
\cos \left( \theta(t)\right) &= x'(t) \quad \text{and} \quad \sin \left( \theta(t)\right) = y'(t).
\end{split}
\end{align}
Clearly, $\varphi(t)$ and $\theta(t)$ are the respective angles, modulo $2\pi$, that $\gamma(t)$ and $\gamma'(t)$ make with the positive $x$-axis.  

Combining \eqref{eqn angles} with \eqref{eqn el} (and suppressing the variable $t$), we find
\begin{align}
\label{eqn: vector field}
\varphi' = \frac{\sin\left( \theta - \varphi\right)}{r},\quad \theta' = 2\frac{(n-1) \cos \theta \cos \varphi - (m-1)\sin\theta \sin \varphi}{r \sin\left(2\varphi\right)}.
\end{align}
Since the set of solution curves of \eqref{eqn el} is invariant under the scaling $\gamma(t) \mapsto c\gamma(t)$ for $c>0$ and the angular variables $\varphi(t), \theta(t)$ are similarly invariant under dilations, it follows that the qualitative behavior of the profile curves is determined by the trajectories of the vector field $V:(0, \pi/2)\times (-\pi, \pi)\rightarrow \R^2$ defined by
\begin{align}
\label{eqn: V}
\begin{split}
V(\varphi, \theta) &= \left( \sin\left(2 \varphi\right)\sin\left( \theta - \varphi\right), 2\left((n-1) \cos \theta \cos \varphi - (m-1)\sin\theta \sin \varphi \right)\right)\\
&:= (V^1, V^2).
\end{split}
\end{align}
Note: this sign convention is consistent with \cite{Alencar1} but has the opposite sign of the vector field $X$ in \cite{Alencar2}.  
On the domain 
\begin{align*}
\Omega = \left\{ (\varphi, \theta) \in (0, \frac{\pi}{2})\times (-\pi, \pi)\right\},
\end{align*}
$V^1$ vanishes precisely on the graphs of the functions
\begin{align*}
\Theta^1_1(\varphi) = \varphi \quad \text{and} \quad \Theta^1_2(\varphi) = \varphi-\pi
\end{align*}
and $V^2$ vanishes precisely on the graphs of the functions 
\begin{align*}
\Theta^2_1(\varphi) = \arctan\left( \frac{n-1}{m-1} \cot \varphi\right) \quad \text{and} \quad \Theta^2_2(\varphi) = \arctan\left( \frac{n-1}{m-1} \cot \varphi\right) - \pi.
\end{align*}
Hence, easily checked monotonicity of the functions $\Theta^1_1, \Theta^1_2, \Theta^2_1$ and $\Theta^2_2$ implies that $V$ has two singular points $p_1:=(\varphi_1, \Theta^1_1(\varphi_1)), p_2:=(\varphi_2, \Theta^2_2(\varphi_2)) \in \Omega$ for numbers $\varphi_1, \varphi_2$ satisfying 
\begin{align*}
\Theta^1_1(\varphi_1) &= \Theta^2_1(\varphi_1)\\
\Theta^1_2(\varphi_2) &= \Theta^2_2(\varphi_2).\\
\end{align*}
\begin{proof}[\textbf{Proof of Theorem \ref{thm: construct}}]  We first prove part (1).  Suppose that $m+n<8$.  
A straightforward singularity analysis reveals that $V$ has focal singularities at $p_1$ and $p_2$.  Further, there is an integral curve 
\[ \psi:(-\infty, \infty)\rightarrow \Omega, \quad \psi(s) = (\varphi(s), \theta(s))\]  of $V$ such that $\lim_{s\rightarrow - \infty}\psi(s) = (0, \pi/2)$ and $\lim_{s\rightarrow \infty} \psi(s) = p_1$ (Lemma 3.1 (i) in \cite{Alencar1} and Corollary 3.11 in \cite{Alencar2})\footnote{In \cite{Alencar2}, Corollary 3.11 (and related results) are stated for $m, n \geq 3$.  However, the proofs remain true when $m, n\geq 2$, as may be seen by examining the proof of Lemma 3.5 \cite{Alencar3}.}.  In particular, as $\psi$ spirals toward $p_1$, there is an increasing sequence $\{s_k\}_{k\in \N}$ such that $\varphi(s_k) = \theta(s_k)$.    Translating this information back to the system \eqref{eqn: vector field}, it follows there is a profile curve $\gamma: [0, \infty) \rightarrow Q$ and an increasing sequence of times $\{t_k\}_{k\in \N}$ such that $\gamma(0)$ intersects the $x$-axis orthogonally and $\varphi(t_k) = \theta(t_k)$.  Define rescaled profile curves 
\begin{align*}
\gamma_k(t): [0, \infty)\rightarrow Q \quad \text{by} \quad \gamma_k(t) = \frac{ \gamma(t)}{|\gamma(t_k)|}. 
\end{align*}
For each $k\in \N$, $\gamma_k(t)$ then satisfies $|\gamma_k(t_k)| = 1$ and $\gamma'_k(t_k) = \gamma_k(t_k)$. 

  It remains to show that $|\gamma_k(t)|<1$ for $t<t_k$ so $\gamma_k$ is in fact a free boundary curve.  For this, we will show 
 \begin{align*} 
 |\gamma(t)|^2 \text{ is monotonic increasing in } t
 \end{align*}
 by a maximum principle argument based on \eqref{eqn: vector field}.  In particular, this will also imply that $\gamma_k$ is embedded.  To carry out the argument, note that 
 \[ \frac{d}{dt}|\gamma(t)|> 0 \quad \text{when} \quad |\varphi(t)-\theta(t)|<\frac{\pi}{2}.\]   On the other hand, by inspection of \eqref{eqn: vector field} it follows that if $\theta(t) - \varphi(t) = \frac{\pi}{2}$ for some $t$, then $\theta'(t) - \varphi'(t)< 0$.  Likewise, if $\varphi(t) - \theta(t) = \frac{\pi}{2}$, then $\varphi'(t) - \theta'(t) > 0$.  Therefore, $|\gamma(t)|$ is monotonic increasing, and hence for each $k\in \N$, there is a unique $t$, namely $t_k$, such that $|\gamma_k(t_k)| = 1$.  Thus, $\gamma_k$ is a free boundary profile curve and this proves the existence of the family $\{ \Sigma_{m, n, k}\}_{k \in \N}$. 
 
 \begin{figure}[h]
\centering
\subfloat[$(m,n) = (4, 2)$]{\includegraphics[width=1.5in]{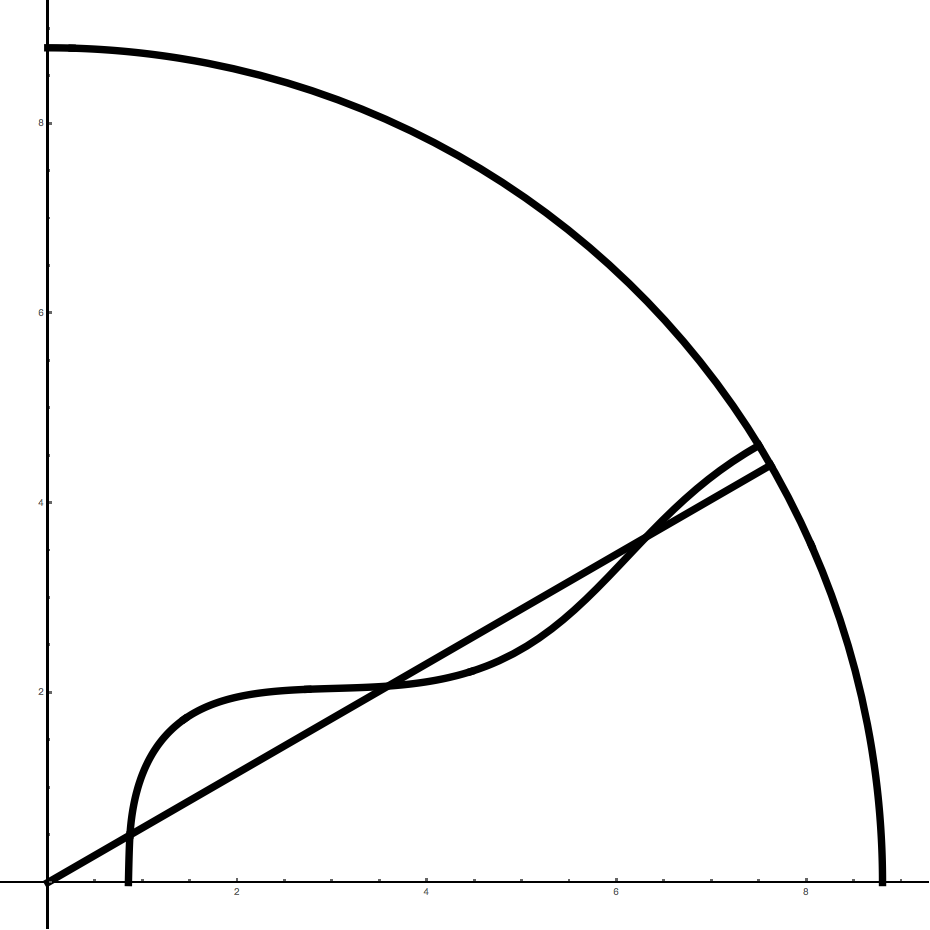}\label{fig:f1}}
\subfloat[$(m,n) = (9, 3)$]{\includegraphics[width=1.5in]{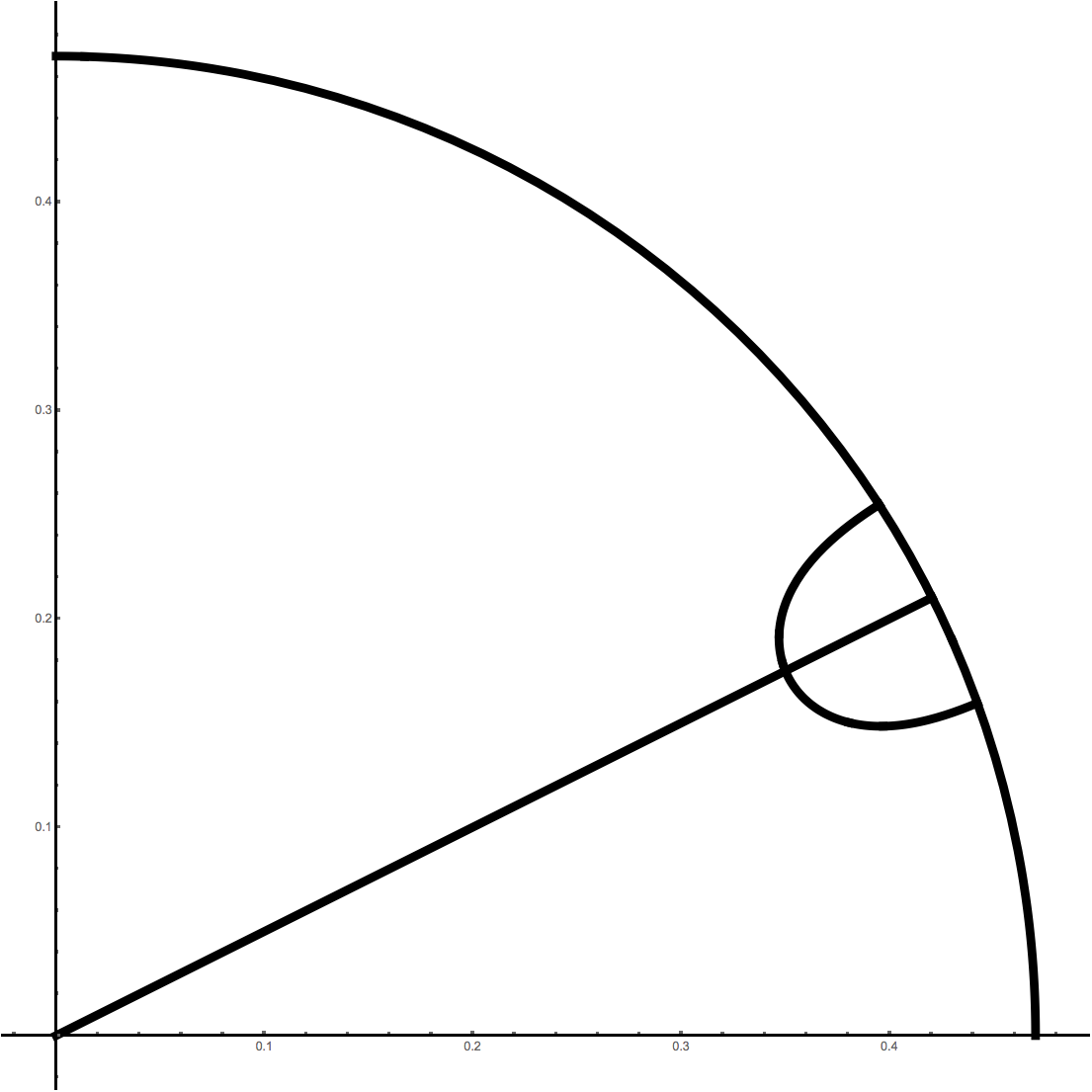}\label{fig:f2}}
\caption{Sketches (in Mathematica) of profile curves $\gamma$ and the lines $\ell_{m,n}$ representative of parts (1) and (2) of Theorem \ref{thm: construct}.  In \ref{fig:f1}, $(m,n) =(4, 2)$ and in figure \ref{fig:f2}, $(m, n) =(9, 3)$.}  
\end{figure}

 By Theorem 1.1 in \cite{Alencar1} and Theorem 1.1, part (3) in \cite{Alencar2}, the profile curve $\gamma$  is asymptotic to the line $\ell_{m,n}$.  Since clearly $\lim_{k\rightarrow \infty}|\gamma(t_k)|= \infty$, it follows that the scalings $\gamma_k$ converge in $C^0$ to $\ell_{m, n}$ on $Q$ and moreover converge in $C^\infty$ to $\ell_{m,n}$ on compact subsets of $Q\setminus \{0\}$. Since the inverse image under $\Pi$ of $\ell_{m,n}$ is the minimal cone $\mathcal{C}_{m,n}$, this completes the proof of part (1).  
 
 We now prove part (2).  Assume now that $m+n\geq 8$.  
Define $\alpha = \tan^{-1}\sqrt{\frac{n-1}{m-1}}$.  Fix $R\in (0, \infty)$, for example $R=\frac{1}{2}$.  For $\epsilon\in [0, \pi]$, let $\gamma_\epsilon$ be the profile curve with initial data
\begin{align}
\label{eqn: profile data}
|\gamma_\epsilon(0)|=R, \quad \varphi(0) = \alpha, \quad \text{and} \quad \theta(0) = \alpha - \epsilon.
\end{align}
As in the proof of part (1) above, $\theta'(t)$ vanishes precisely on the graphs of the functions $\Theta^2_1$ and $\Theta^2_2$.  In particular, by this and inspection of \eqref{eqn: vector field}, it follows that
\begin{align}
\label{eqn: flow monotonicity}
\begin{split}
\theta'&> 0 \quad \text{for} \quad \varphi\in (0, \pi/2), \quad \Theta^2_2(\varphi)< \theta< \Theta^2_1(\varphi),\\
\varphi'&< 0 \quad \text{for} \quad \varphi \in (0, \pi/2), \quad -\pi < \theta- \varphi <0.  
\end{split}
\end{align}
Hence, there are unique times $t_-(\epsilon)<0<t_+(\epsilon)$ such that
\begin{align*}
\varphi(t_-) - \pi = \theta(t_-) \quad \text{and} \quad \varphi(t_+) = \theta(t_+).
\end{align*}
Moreover, $\theta'(t)>0$ and $\varphi'(t)< 0$ for $t\in [t_-, t_+]$.  
Let $0<\delta< 1/2$.  It is clear that the points $\{\gamma_\epsilon(t_-): \epsilon \in (\pi - 1/2, \pi)\} $ foliate an open ray of the line 
\[ \{(\varphi, \theta): \varphi \in (\alpha, \alpha+\delta),  \theta = \varphi - \pi\}\] 
 when $\delta$ is sufficiently small.  Hence, the images of the curves 
 \[\{ \gamma_\epsilon(t): t\in [t_-(\epsilon), 0], \epsilon \in (0, \pi-1/2)\} \] 
 are contained in the complement of fixed compact sets containing the graphs of $\Theta^2_1$ and $\Theta^2_2$.  This implies that for such curves $\gamma_\epsilon$, $(n-1) \cos \theta \cos \varphi - (m-1)\sin\theta \sin \varphi$ is uniformly  bounded away from $0$ for $t$ near $t_-(\epsilon)$.  In particular, this is true as $\epsilon \searrow 0$.

By the monotonicity of $\varphi$ and $\theta$, it follows that $r(t_+(\epsilon))>R$ as $\epsilon\searrow 0$.  On the other hand, we claim that $r(t_-(\epsilon))\rightarrow 0$ as $\epsilon\searrow 0$.  By the monotonicity of $\varphi$ and $\theta$, there is a unique $t_0(\epsilon) \in (t_-(\epsilon), 0)$ such that $\varphi(t_0) - \theta(t_0) = \frac{\pi}{2}$.  In particular, $r'(t)<0$ for $t<t_0$, $r'(t)>0$ for $t>t_0$ and $r$ attains a global minimum at $t=t_0$. 

By smooth dependence on ODE solutions based on initial conditions, $\gamma_\epsilon$ converges to $\ell_{m,n}$ as $\epsilon\searrow 0$.  Therefore, there is a sequence $\epsilon_n \searrow 0$ such that $|\gamma_{\epsilon_n}(t_0)|<\frac{1}{n}$.  By this and the preceding, it follows from Equation \eqref{eqn: vector field} that $\lim_{n\rightarrow \infty} \theta'(t) = 0$ uniformly for $t\in (t_-(\epsilon_n), t_0(\epsilon_n))$.  Since $\gamma_{\epsilon_n}$ is unit speed parametrized, this implies that $\epsilon_n$ can be chosen such that $|\gamma_{\epsilon_n}(t_-)|< \frac{2}{n}$.  Therefore, $|\gamma_{\epsilon}(t_-)|<|\gamma_{\epsilon}(t_+)|$ as $\epsilon\searrow 0$.  On the other hand, by adapting the argument above as $\epsilon \nearrow \pi$, we find that $|\gamma_{\epsilon}(t_-)|>|\gamma_{\epsilon}(t_+)|$ as $\epsilon \nearrow  \pi$. By continuity, it follows that there is an $\bar{\epsilon} \in (0, \pi)$ such that $|\gamma_{\bar{\epsilon}}(t_-)| = |\gamma_{\bar\epsilon}(t_+)|$.  It follows that the rescaled curve
\begin{align*}
\gamma_{m,n}: [t_-, t_+]\rightarrow Q \quad \text{defined by} \quad \gamma_{m,n} =  \frac{ \gamma_{\bar\epsilon}(t)}{|\gamma_{\bar\epsilon}(t_-)|}
\end{align*}
is a free boundary profile curve, and the proof is complete.  
\end{proof}

\begin{remark}
When $m+n\geq 8$, the vector field $V$ has nodal singularities.  It is easy to check that in this case, the integral curve of $V$ defined in an analogous way to $\psi$ from the proof of Theorem \ref{thm: construct} does not intersect the line $\{\varphi = \theta\}$ (see also \cite{Alencar2}, Proposition 4.4, parts (1) and (3)). Therefore, there are no free boundary minimal surfaces in $B^{m+n}(1)$ of the type from Theorem \ref{thm: construct}, part (1) when $m+n\geq 8$.  When $m+n<8$, the oscillatory behavior of the trajectories of $V$ make the continuity argument in the proof of Theorem \ref{thm: construct}, part (2) break down.  
\end{remark}

\section{Torques and Balancing}
\label{forces}
We recall some notions regarding fluxes and force balancing for orientable minimal surfaces $\Sigma^n \subset \R^{n+1}$ (cf. \cite{Osserman}, p. 31).  Let $\Sigma$ be such a surface, $\eta$ be the outward pointing unit conormal vector field on $\partial \Sigma$, and let $K$ be a Killing field on $\R^{n+1}$.  A consequence of the First Variation Formula is the \emph{Balancing Formula}: 
\begin{align}
\label{balancing}
\int_{\partial \Sigma} \!  K\cdot  \eta  = 0.
\end{align}

Suppose $\sigma \subset \partial \Sigma$ is a boundary component.  Recall the \emph{Flux} about $\sigma$, $\mathcal{F}(\sigma)$, is defined by 
\[ \mathcal{F}(\sigma) = \int_{\sigma}\! \eta\]
and depends only on the homology class $[\sigma]$ of $\sigma$.  

If $\Sigma \subset \R^3$, the \emph{Torque}  $\T(\sigma)$ about $\sigma$ is defined by
\[ \T(\sigma) = \int_{\sigma}\! X \times \eta\]
where $X$ is the position vector and $\times$ represents the cross product.  For a vector $v \in \R^3$ we denote $K_{v}$ the counterclockwise rotation about $v$ defined by
\[ K_{v}(X) = v\times X.\]
If $W$ is another base point, we similarly denote $K_{v, W}$ the counterclockwise rotation about $v$ based at $W$, defined by
\[ K_{v, W}(X) = v \times (X-W).\] 
Observe the Torque satisfies
\begin{align}
\label{torque def}
\T(\sigma)\cdot v = \int_{\sigma}\! (X\times \eta)\cdot v = \int_{\gamma} \!(v\times X)\cdot \eta = \int_{\sigma} \!K_{v} \cdot \eta
\end{align}
so $\T(\sigma)$ also depends only on the homology class $[\sigma]$ of $\sigma$.  Similarly we have the torque measured from a base point $W$ is defined by
\begin{align}
\label{resultant torque}
\T_W( \sigma) = \int_{\sigma} \!(X - W)\times \eta = \T(\sigma) - W\times \F(\sigma).
\end{align}

\section{Properties of the $n$-catenoid}\label{cat properties}
In this section we borrow some notation from \cite{Pacard 1}.   We define an \emph{$n$-catenoid} to be a complete, nonplanar minimal hypersurface of revolution in $\R^{n+1}$.  We shall only consider hypersurfaces of revolution about the $x^{n+1}$-axis.  Locally, such surfaces may be parametrized by
\[ (z,r(z)\omega ) \] 
where $\omega$ locally parametrizes $S^{n-1}$ and the profile function $r(z)$ satisfies the differential equation
\begin{equation}
\label{profile}
 \ddot{r} r -(n-1)(1+\dot{r}^2) = 0
 \end{equation}
 where $z$ is the usual coordinate on the $x^{n+1}$-axis and $\cdot$ denotes differentiation with respect to $z$. 

If we normalize so that $r(0) = 1$ and that $\dot{r}(0) = 0$, $r$ also satisfies
\begin{equation}
\label{fop}
1+\dot{r}^2 = r^{2n-2}.
\end{equation}
It is worth remarking that any solution of \eqref{profile} may be obtained by dilating and translating a solution of \eqref{fop}.  By combining Equations \eqref{profile} and \eqref{fop}, we obtain
\begin{equation}
\label{fop2}
\ddot{r} = (n-1)r^{2n-3}. 
\end{equation}
For $n=2$, the solution of \eqref{profile} with $r(0)=1$ and $r'(0) = 0$ is $r(z)=\cosh(z)$ and is in particular defined for all $z$.  However, for $n>2$, the solution of \eqref{profile} with the same initial data is defined in an interval $(-T, T)$, where
\[ T = \int_1^\infty \!\frac{dz}{(z^{2n-2}-1)^\frac{1}{2}}.\]

For such a profile function $r(z)$ we call $\gamma(z)  = (z, r(z))$ the associated \emph{profile curve}. 

The critical catenoid is a scaling of the catenoid $\sqrt{x^2+y^2} = \cosh z$ that satisfies the free boundary condition.  Specifically, its profile curve is given by
\begin{align*}
\gamma(z) = (z, r(z) ) = (z, \frac{1}{\tau} \cosh(\tau z))
\end{align*}
where $\tau = \sigma \cosh(\sigma)$ and $\sigma$ is the positive solution to the equation $z = \coth z$.




\section{Proof of theorem \ref{thm: catenoid}} 
\label{main proof}

Suppose $\Sigma^n \subset B^{n+1}(1)$ is a free boundary catenoid with axis of rotation $\ell_{\Sigma}$.  After possibly applying a rotation, we may suppose that $\ell_\Sigma$ is parallel to the $x^{n+1}$-axis.

\begin{lemma}
\label{lemma: coaxial}
$0\in \ell_\Sigma$. 
\end{lemma}
\begin{proof}
We first consider the case when $n=2$.  Let $\{e_1, e_2, e_3\}$ be the standard orthonormal frame on $\R^3$.  $\partial \Sigma$ consists of two circles in planes parallel to $\{z=0\}$; let $\sigma$ be the component with higher $z$-coordinate.  
Recalling \eqref{torque def}, we have
\begin{align*}
\T(\sigma) &=  \left(\int_\sigma \! K_{e_1} \cdot \eta, \int_\sigma \! K_{e_2}\cdot  \eta, \int_\sigma \!K_{e_3}\cdot  \eta\right)\\
&=(0, 0, 0)
\end{align*}
since $K_{e_i}$ is tangent to the sphere and the free boundary condition implies $K_{e_i}\cdot \eta = 0$ for $i=1, 2,3$. 

Let $\sigma_0$ be the ``waist'' circle of $\Sigma$ oriented so that $\eta = (0, 0, 1)$.  Since $\sigma$ is homologous to $\sigma_0$, it follows that
\[\F(\sigma) = \F(\sigma_0) = (0, 0, 2\pi \tau)\]
where $\tau$ is the radius of the circle $\sigma_0$.  Let $W\in \text{span}\{e_1, e_2\}$ be in $\ell_\Sigma$.  Since $\sigma$ is homologous to $\sigma_0$ and $\T_W$ is a homology invariant, we similarly compute
\begin{align*}
\T_W(\sigma) &= \T_W(\sigma_0)\\
&= \left(\int_{\sigma_0} \!K_{e_1, W} \cdot \eta, \int_{\sigma_0}\! K_{e_2, W}\cdot  \eta, \int_{\sigma_0} \!K_{e_3, W}\cdot  \eta\right)\\
&= (0,0,0)
\end{align*}
since  $K_{e_3, W}\cdot \eta = 0$ pointwise on $\sigma_0$ and the the latter two integrals vanish by symmetry.  

On the other hand, Equation \eqref{resultant torque} implies 
\begin{align*}
\T_W(\sigma) &= \T(\sigma) - W\times \F(\sigma).
\end{align*}
This together with the fact that $\F(\sigma)$ is a vertical vector implies that $W$ is also a vertical vector, which proves the lemma when $n=2$.  

When $n>2$, consider a three dimensional subspace $U$ spanned by $e_1: = W$, $e_3: =\ell_\Sigma$ and an arbitrary unit vector $e_2$ mutually orthogonal to $e_1$ and $e_3$.  Then defining $\T(\sigma)$ and $\T_W(\sigma)$ by the same formula as before where $X$ and $\eta$ this time denote the projections of the corresponding vectors to $U$, the proof follows in an analogous way.  
\end{proof}


\begin{remark}
\label{rmk 1}
In light of Lemma \ref{lemma: coaxial}, after dilations and a possible reflection, we may restrict our considerations to translated profile curves of the form
\[ \gamma(z, c) = (z, r(z-c))\]
where $c\geq 0$ and $r(z)$ satisfies $r(0)=1$, $\dot{r}(0) = 0$, and equation \eqref{fop}.  For convenience of notation, we denote differentiation with respect to $c$ and $z$ by $'$ and $\cdot$, respectively.  The catenoid corresponding to $\gamma$ is a free boundary catenoid in $B^{n+1}(R)$ if for each value $z_i$ where $|\gamma(z_i, c)| = R$, $\gamma(z_i, c)$ and $\dot{\gamma}(z_i,c)$ lie along the same line. 
\end{remark}

\begin{lemma}\label{general technical lemma}
Suppose $n\geq 2$.  There are two real values $z_1(c)<0<z_2(c)$ such that $\gamma(z)$ and $\dot{\gamma}(z)$ point along the same line.  Moreover, for $i=1, 2$ we have
\begin{enumerate}
\item $\frac{d}{dc} z_i(c)= 1-\frac{\dot{r}^2}{(n-1)(1+\dot{r}^2)}>0$.
\item $\frac{d}{dc} (z_i-c)=-\frac{\dot{r}^2}{(n-1)(1+\dot{r}^2)}<0$.
\end{enumerate}
\end{lemma}
\begin{proof}
First fix $c\geq 0.$  Using 
\begin{align*}
\gamma(z) = (z, r(z-c)) \quad \text{and} \quad \dot{\gamma}(z)  = (1, \dot{r}(z-c))
\end{align*}
we find that $\dot{\gamma}(z)$ and $\gamma(z)$ point along the same line if and only if
\[ z = \frac{r}{\dot{r}} = \frac{ r}{(r^{2n-2}-1)^{1/2}}.\]
For $z\geq c$, define $h_+(z) = z - \frac{r}{(r^{2n-2} -1)^{1/2}}$, and compute
\begin{align*}
\dot{h}_+(z) &= (n-1) \left( 1+ \frac{1}{\dot{r}^2}\right). 
\end{align*}
Hence $h_+(z)$ is monotonically increasing where it is smooth.  It is clear that $h_+(z)> 0$ as $c\nearrow T$ and
\[ \lim_{z \searrow c} h_+(z) = -\infty\]
so there is a unique solution $z_2(c)> 0$ satisfying $h_+(z_2) = 0$.
Similarly, for $z<c$, define $h_-(z) =  z+ \frac{ r}{(r^{2n-2}-1)^{1/2}}.$
Using analogous arguments, we see there is a unique $z_1(c)< 0$ such that $h_-(z_1) = 0$.  

Now differentiate the equation $z_i(c)\dot{r}(z_i-c) = r(z_i-c)$ implicitly with respect to $c$ and use \eqref{fop} and \eqref{fop2} to conclude
\begin{align*}
\frac{d}{dc} z_i(c)  = \frac{1}{n-1}\left(n-2 - \frac{1}{r^{2n-2}}\right) = 1-\frac{\dot{r}^2}{(n-1)(1+\dot{r}^2)}>0
\end{align*}
which proves (1).  (2) then follows easily since 
\begin{align*}
\frac{d}{dc}(z_i - c)  = -\frac{\dot{r}^2}{(n-1)(1+\dot{r}^2)}<0.
\end{align*}  
\end{proof}


\begin{lemma}\label{first derivative positive estimate}
The positive solution of the equation $z_0 = \frac{r(z_0)}{\dot{r}(z_0)}$ satisfies $r(z_0)\geq n^\frac{1}{2n-2}$.
\end{lemma}
\begin{proof}
 An explicit calculation using $\gamma = (z, \cosh z)$ shows the conclusion holds when $n=2$.  For $n\geq 3$, we introduce a change of parameters which appears in \cite{Pacard 1} which will enable a somewhat explicit estimate. 
We use the new parameter $t$ where
\begin{align}
\label{general profile curve}
(z, r(z)) = ( \psi(t), \phi(t) )
\end{align} 
where
\begin{align}
\label{profile eqns}
\phi(t) = \cosh( (n-1) t)^\frac{1}{n-1} \quad \text{and} \quad \psi(t) = \int_{0}^{t} \!\phi^{2-n}(s) ds.
\end{align}
The condition that $\gamma(t)$ and $\gamma'(t)$ point along the same line is equivalent to 
\begin{equation}
\label{derivative estimate}
  \sinh((n-1)t) \int_{0}^{t}\! \cosh( (n-1)s )^\frac{2-n}{n-1} ds  = \cosh( (n-1)t)^\frac{1}{n-1}.
\end{equation}
Let $t_0$ denote the unique positive value of $t$ where the above equation is satisfied.  

There is a unique $v> 0$ such that $\cosh( (n-1) v) = \sqrt{n}$, and by direct calculation
\[ v= \frac{\cosh^{-1}(\sqrt{n})}{n-1} = \frac{1}{2(n-1)} \log( \sqrt{n}+\sqrt{n-1}).\]
Since the left hand side of Equation \eqref{derivative estimate} is a monotone increasing function of $t$ and eventually dominates the right hand side, the lemma will follow if we can show 
\begin{align}
\label{claim 1}
\sinh((n-1)v) \int_{0}^{v}\! \cosh( (n-1) s)^\frac{2-n}{n-1} ds < n^\frac{1}{2n-2},
\end{align}
for then at $t_0$ we will have 
\[ r(z_0) = \phi(t_0) = \cosh((n-1)t)^\frac{1}{n-1} > n^\frac{1}{2n-2}.\]
For $n\geq 3$, we have $n-2< 0$, so $\cosh( (n-1) t)^\frac{n-2}{n-1}< 1$.  Hence 
\begin{align*}
\sinh((n-1)v) \int_{0}^{v} \!\cosh( (n-1) s)^\frac{2-n}{n-1} ds
&< \sqrt{n-1}  \frac{1}{2(n-1)} \log( \sqrt{n}+\sqrt{n-1})\\
&< \frac{1}{2\sqrt{n-1}} \log(2\sqrt{n}).
\end{align*}
It is easy to check that this is less than $n^\frac{1}{ 2n-2}$ for $n\geq 3$.  This verifies \eqref{claim 1}. 
\end{proof}

\begin{proof}[\textbf{Proof of Theorem \ref{thm: catenoid}}]
 By Lemma \ref{lemma: coaxial} and Remark \ref{rmk 1}, we may suppose there is a $c\geq 0$ such that $\Sigma$ has the profile curve \[ \gamma(z, c) = (z, r(z-c)).\]
Define $f_i(c) =| \gamma(z_i(c))|^2$.  By Lemma \ref{general technical lemma} and Remark \ref{rmk 1}, Theorem \ref{thm: catenoid} will follow if we can show 
\begin{align}
\label{eqn: main claim}
f_1(c)>f_2(c) \quad  \text{for}\quad c>0.   
\end{align}
We will prove \eqref{eqn: main claim} by a maximum principle argument.  
 After computation using \eqref{fop}, \eqref{fop2}, and Lemma \ref{general technical lemma},
\begin{align*}
\frac{1}{2} f'_i(c) &= z_i z'_i + r(z_i -c) \dot{r}(z_i -c)(z_i-c)'\\
&= \frac{z_i}{n-1}\left( n - r^{2n-2}\right).
\end{align*}
By Lemma \ref{first derivative positive estimate}, $f_1'(0)> 0$ and $f_2'(0)< 0$. 
We compute the second derivatives in the same way: 
\begin{align*}
\frac{1}{2}f^{\prime \prime}_i(c) &= \frac{z'_i}{n-1}(n-r^{2n-2}) + \frac{z_i}{n-1}\left( -(2n-2)r^{2n-3} \dot{r}\frac{d}{dc}(z_i-c)\right)\\
&= \frac{1}{n-1}\left( (2n-2)r^{2n-2} + z_i'( n-(2n-1))r^{2n-2}\right)
\end{align*}
and after further calculation,
\begin{align*}
\frac{n-1}{2}f^{\prime \prime}_i(c) = n-3-\frac{2}{n-1}+\frac{n}{n-1}\left(r^{2n-2}(z_i-c)+\frac{1}{r^{2n-2}(z_i-c)}\right).
\end{align*}

Observe that $f^{\prime \prime}_1(0) = f^{\prime \prime}_2(0)$.  We compute the difference of third derivatives as follows:
\begin{align*}
\frac{n-1}{2}\left(f^{(3)}_1(c)-f^{(3)}_2(c) \right) &= \frac{n}{n-1}\left( r^{2n-2}(z_1-c) - r^{2n-2}(z_2-c)\right)'\\
& - \frac{n}{n-1}\left(\frac{1}{r^{2n-2}(z_2-c)}-\frac{1}{r^{2n-2}(z_1-c)}\right)'
\end{align*}

Using Lemma \ref{first derivative positive estimate}, we estimate the first group of terms as follows:
\begin{eqnarray*}
\frac{n}{n-1}\left( r^{2n-2}(z_1-c) -r^{2n-2}(z_2-c) \right)' &\geq& \frac{n}{n-1}(2n-2) r^{2n-3}(z_1-c) \dot{r}(z_1-c) \left(1-\frac{1}{r^{2n-2}(z_1-c)}\right)\\
&\geq& \frac{n}{n-1}2(n-1)n^{\frac{2n-3}{2n-2}} \frac{(n-1)^{3/2}}{n}.\\
&=& 2n^\frac{2n-3}{2n-2}(n-1)^{3/2}.
\end{eqnarray*}
For the second group of terms, we estimate
\begin{eqnarray*}
\frac{n}{n-1}\left(\frac{1}{r^{2n-2}(z_2-c)}-\frac{1}{r^{2n-2}(z_1-c)}\right)'=\\
\frac{2n}{n-1}\left(\frac{1}{r^{2n-1}(z_2-c)}\frac{\dot{r}(z_2-c)^2}{\dot{r}(z_2-c)^2+1} -  \frac{1}{r^{2n-1}(z_1-c)}\frac{\dot{r}(z_1-c)^2}{\dot{r}(z_1-c)^2+1} \right)\\
 \leq 3. 
\end{eqnarray*}
From these estimates, it easily follows that $f^{(3)}_1(c)-f^{(3)}_2(c) \geq  0$ for all $n\geq 2$.  This implies $f^{\prime \prime}_1(c) \geq f^{\prime \prime}_2(c)$ for all $c\geq 0$; since $f'_1(0)> 0$ and $f'_2(0)<0$, it follows that $f_1(c)>f_2(c)$ for $c>0$.  Hence \eqref{eqn: main claim} holds and Theorem \ref{thm: catenoid} follows.  
\end{proof}

\begin{remark}
\label{rmk: crit cat}
In the case where $\Sigma^2 \subset B^3(1)$ is a free boundary catenoid, we can give another proof of Lemma \ref{lemma: coaxial}.  It follows from
\begin{lemma}
Suppose $\Sigma^2\subset B^3(1)$ is a free boundary minimal surface.  Each component of $\partial \Sigma$ is a line of curvature of $\Sigma$.
\end{lemma}
\begin{proof}
Let $\gamma(t)$ be a local unit speed parametrization of a component of $\partial \Sigma$  and let $n$ be a local unit normal field on $\Sigma$.  By the free boundary condition it is possible to orient $\gamma$ such that for each $t\in I$, $\{\gamma, \dot{\gamma}, n\}$ is an orthonormal frame for $T_{\gamma(t)}\R^3$ and $\{\gamma, \dot{\gamma}\}$ is an orthonormal frame for $T_{\gamma(t)} \Sigma$.  Then compute
\begin{align*}
0 =\dot{\gamma}\langle n, \gamma\rangle &= \langle \nabla_{\dot{\gamma}} n, \gamma\rangle + \langle n, \nabla_{\dot{\gamma}}\gamma\rangle\\
&= \langle \nabla_{\dot{\gamma}} n, \gamma\rangle+\langle n , \dot{\gamma}\rangle\\
&= \langle \nabla_{\dot{\gamma}} n, \gamma\rangle.
\end{align*}
Since $\Sigma$ is a surface, this implies $\dot{\gamma}$ is an eigenvector of the Weingarten map, which proves the Lemma.  
\end{proof}
Since the compact lines of curvature on any catenoid $\Sigma\subset \R^3$ are circles centered on $\ell(\Sigma)$, elementary geometry implies $0\in \ell(\Sigma)$.  
\end{remark}


\end{document}